\documentclass[leqno,11pt]{amsart}
\usepackage{amssymb, amsmath,amsmath,latexsym,amssymb,amsfonts,amsbsy, amsthm}
\usepackage{color}

\let\pa=\partial

\let\f=\frac

\let\om=\omega

\let\Om=\Omega

\let\ve=\varepsilon
\let\pa=\partial
\let\va=\varphi

\def\cN{{\cal N}}

\def\cN{{\mathcal N}}

\def\na{\nabla}

\newcommand{\beq}{\begin{equation}}
\newcommand{\eeq}{\end{equation}}
\newcommand{\ben}{\begin{eqnarray}}
\newcommand{\een}{\end{eqnarray}}
\newcommand{\beno}{\begin{eqnarray*}}
\newcommand{\eeno}{\end{eqnarray*}}

\newtheorem{theorem}{Theorem}[section]

\newtheorem{lemma}[theorem]{Lemma}
\newtheorem{proposition}[theorem]{Proposition}

\newtheorem{Theorem}{Theorem}[section]

\newtheorem{Remark}[Theorem]{Remark}

\newcommand{\ud}{\mathrm{d}}

\newcommand{\nn}{\mathbf{n}}

\newcommand{\mm}{\mathbf{m}}

\newcommand{\II}{\mathbf{I}}

\newcommand{\BP}{\mathbb{P}}
\newcommand{\BQ}{\mathcal{Q}}

\newcommand{\BR}{{\mathbb{R}^2}}

\begin{document}
\title{Stability of Half-Degree Point Defect profiles for 2-D Nematic Liquid Crystal}

\author{Zhiyuan Geng}
\address{Courant Institute of Mathematical Sciences, New York University}
\email{zg574@nyu.edu}

\author{Wei Wang}
\address{Department of Mathematics, Zhejiang University, 310027, Hangzhou, P. R. China}
\email{wangw07@zju.edu.cn}

\author{Pingwen Zhang}
\address{School of Mathematical Sciences, Peking University, 100871, P. R. China}
\email{pzhang@pku.edu.cn}

\author{Zhifei Zhang}
\address{School of Mathematical Sciences, Peking University, 100871, P. R. China}
\email{zfzhang@math.pku.edu.cn}

\begin{abstract}
In this paper, we prove the stability of half-degree point defect profiles
in $\mathbb{R}^2$  for the nematic liquid crystal within Landau-de Gennes model.
\end{abstract}

\maketitle

\section{Introduction}

Defects in liquid crystal are known as the places where the degree of symmetry of the nematic order increases so that
the molecular direction cannot be well defined. The most striking feature
of liquid crystal is a variety of visual defect patterns. Predicting the profiles of defect as well as stability
is thus of great practical importance and theoretical interest. We mention some works \cite{BP, KL, M, TK} on the defects based
on the topological properties of the order parameter manifolds.

There exist three commonly used continuum theories describing the nematic liquid crystal: Oseen-Frank model, Ericksen model
and Landau-de Gennes model. In the Oseen-Frank model, the state of nematic liquid crystals is described by a unit-vector filed which represents the mean local orientation of molecules, and defects are interpreted as all singularities of this vector field \cite{HKL,He, E, LL}. However, the core structure of defects in nematic liquid crystals, such as the disclination lines observed in experiments, cannot be represented by the usual director field and requires description by Landau-de Gennes model \cite{dGP}. In this model, the state of nematic liquid crystals is described by a $3\times3$ order tensor $Q$ belonging to
\begin{align*}
\BQ=\Big\{Q: Q\in\mathbb{M}^{3\times 3},\,Q=Q^T,\, \mathrm{tr} Q=0\Big\}.
\end{align*}
For $Q\in\BQ$, one can find $s,b\in \mathbb{R}, \nn,\mm\in \mathbb{S}^2$ with $\nn\cdot\mm=0$ such that
\begin{equation}
\nonumber Q=s(\nn\otimes \nn-\frac{1}{3}\II)+b(\mm\otimes\mm-\frac{1}{3}\II),
\end{equation}
where $\mathbf{I}$ is $3\times3$ identity matrix. The local physical properties of nematic liquid crystals depend on the degree of symmetry of order tensor $Q$.
Specifically, there are three different states:
\begin{enumerate}
  \item $s=b=0$, which describes the isotropic distribution;
  \item $s\neq 0,b=0$, which corresponds to the uniaxial distribution;
  \item $s\neq 0, b\neq 0$, which describes the biaxial distribution.
\end{enumerate}

Configuration of nematic liquid crystals corresponds to local minimizers of Landau-de Gennes energy functional, whose
simplest form is given by
\begin{equation}\label{energy:LG}
\mathcal{F}_{LG}[Q]=\int_{\Om}\Big\{\f{L}{2}|\nabla Q(x)|^2+f_B(Q(x))\Big\}\ud x,
\end{equation}
where $L>0$ is a material-dependent elastic constant, and $f_B$ is the bulk energy density,  which can be taken as follows
\begin{equation}
\nonumber f_B(Q)=-\f{a^2}{2}\mathrm{tr}(Q^2)-\f{b^2}{3}\mathrm{tr}(Q^3)+\f{c^2}{4}\mathrm{tr}(Q^2)^2,
\end{equation}
where $a^2$,$b^2$,$c^2$ are material-dependent and non-zero constants, which may depend on temperature. A well-known fact is that $f_B(Q)$ attains its minimum on a manifold $\cN$  given by
\beno
\cN=\Big\{Q\in \BQ:Q=s^{+}(\nn\otimes \nn-\f{1}{3}\II),\, \nn\in \mathbb{R}^3,|\nn|=1\Big\},
\eeno
where $s^{+}=\f{b^2+\sqrt{b^4+4a^2c^2}}{4c^2}$. It is easy to see that $\cN$ is a smooth submanifold of $\BQ$,
homemorphic to the real projective plane $\mathbb{R}\BP^2$, and contained in the sphere $\Big\{Q\in \BQ:|Q|=\sqrt{\f{2}{3}}s^{+}\Big\}$.
Critical points of Landau-de Gennes functional satisfy the Euler-Lagrange equation
\begin{equation}\label{eq:EL}
 L\Delta Q=-a^2Q-b^2(Q^2-\frac{1}{3}|Q|^2I)+c^2Q|Q|^2.
\end{equation}

The Landau-de Gennes energy (\ref{energy:LG}) and Euler-Lagrange equation (\ref{eq:EL}) are widely
used to study the  behavior of defects, see \cite{BaP, C, GM, MZ}
and references therein. However, there still exist many challenging problems in understanding the
mechanism which generates defects and predicting their profiles as well as stability, see \cite{HQZ}
for many conjectures. The radial symmetric solution in a ball or in $\mathbb{R}^3$, named
hedgehog solution, is regarded as a potential candidate profile for the isolated point defect in 3-D region. The property and stability
of this solution are well studied and it is shown that the radial symmetric solution are not stable for large $a^2$
and stable for small $a^2$ \cite{INSZ1}. We also refer \cite{RV, Ma, La, INSZ0} and references therein for related works.

In this paper, we are concerned with  a class of point defects in $\mathbb{R}^2$,
which correspond to ``radial" solutions of  the Euler-Lagrange equation (\ref{eq:EL}).
Here ``radial" means that the eigenvectors of $Q$ don't change along the radial direction. Precisely speaking,
we study the solution with the form
\begin{align}\label{prsolu}
    Q(r,\varphi)
      &=u(r)F_1+v(r)F_2,
\end{align}
where $(r,\varphi)$ is the polar coordinate in $\mathbb{R}^2$, and
\begin{align}
  F_1=2\nn\nn-\II_2=\left(
        \begin{array}{ccc}
          \cos k\varphi & \sin k\varphi & 0 \\
          \sin k\varphi & -\cos k\varphi & 0 \\
          0 & 0 & 0 \\
        \end{array}
      \right),\quad
  F_2=3\mathbf{e}_3\otimes\mathbf{e}_3-\II=\left(
        \begin{array}{ccc}
          -1 & 0 & 0 \\
          0 & -1 & 0 \\
          0 & 0 & 2 \\
        \end{array}
      \right).
\end{align}
The boundary condition on these solutions is taken to be
\begin{align}\label{boundary}
\lim_{r\to+\infty}Q(r,\varphi)=s_+(\nn(\varphi)\otimes \nn(\varphi)-\frac{1}{3}\II),\quad \nn(\varphi)=(\cos{\frac{k}{2}\varphi},\sin{\frac{k}{2}\varphi},0),~~
\end{align}
which has degree $\frac k 2$ about origin as an $\mathbb{R}\BP^2$-valued map. Here $k\in\mathbb{Z}\setminus\{0\}$.
Note that if we assume the invariance of $Q$ along the defect line,
then disclination line in 3-D domain can be ideally treated as a point defect in 2-D domain.

In \cite{INSZ2}, Ignat, Nguyen, Slastikov and Zarnescu proved the existence of the radial solution for any non-zero integer $k$. Moreover, the solution is also a local minimizer of the reduced functional (\ref{energy:radi}).  An important question is whether the radical solution they constructed is a local minimizer of the energy $\mathcal{F}_{LG}$. This problem was also partially answered in \cite{INSZ2}, where the instability result is proved for $|k|>1$.
However, the question of whether the $k$-radially symmetric solutions (\ref{prsolu}) subject to (\ref{boundary}) for $k=\pm 1$ are stable remains open.

The goal of this paper is to give a positive answer to this question. Precise result will be stated in next section.
We remark that this problem is somewhat analogous to the stability of radial solutions of the Ginzburg-Landau equation(see \cite{Lieb, Miron, LiT, GS, PFK} for example).

\setcounter{equation}{0}
\section{The stability of radially symmetric solution with $k=\pm 1$}

We make the following rescaling
\beno
\widetilde{Q}=\f{c^2}{b^2}Q,\quad \widetilde{x}=\sqrt{\f{2}{L}}\f{b^2}{c},
\eeno
and let $t=\f{a^2c^2}{b^4}$. Then Landau-de Gennes energy functional (\ref{energy:LG}) is rescaled into the form(drop the tildes):
\begin{equation}
\nonumber \mathcal{F}_{LG}[Q]=\int_{\mathbb{R}^2}\Big\{\f{1}{2}|\nabla Q(x)|^2-\f{t}{2}\mathrm{tr}(Q^2)-\f{1}{3}\mathrm{tr}(Q^3)+\f{1}{4}\mathrm{tr}(Q^2)^2\Big\}\ud x.
\end{equation}
Therefore, without loss of generality, we may take $L=b=c=1$ and $a^2=t>0$.

In such case, substituting (\ref{prsolu}) into (\ref{eq:EL}), $(u,v)$ satisifes
the following ODE system(see \cite{DRSZ, HQZ}):
\begin{eqnarray} \label{ODE}
\left\{
\begin{array}{l}
  u''+\frac{u'}{r}-\frac{k^2}{r^2}u=u[-t+2v+(6v^2+2u^2)] , \\
  v''+\frac{v'}{r}=v[-t-v+(6v^2+2u^2)]+\frac{1}{3}u^2,
\end{array}\right.
\end{eqnarray}
together with the boundary conditions
\begin{equation}\label{bdy2}
u(0)=0,\quad v'(0)=0,\quad u(\infty)=\frac{s^{+}}{2},\quad v(+\infty)=-\frac{s^{+}}{6},
\end{equation}
where $s^+=\f{1+\sqrt{1+24t}}{4}$. The system (\ref{ODE}) can be also viewed as the Euler-Lagrange equation of the functional
\begin{align}
  \nonumber E(u,v)=&\int_0^\infty\Big\{\frac{1}{2}\big[2(\frac{\partial u}{\partial r})^2+6(\frac{\partial v}{\partial r})^2+\frac{2k^2u^2}{r^2}\big]-\frac{t}{2}(2u^2+6v^2)\\ \label{energy:radi}
   &\qquad\qquad\qquad-2v(v^2-u^2)+\frac{1}{4}(2u^2+6v^2)^2 \Big\}rdr,
\end{align}
which is a reduced form of (\ref{energy:LG}) under the assumption (\ref{prsolu}).\medskip

Let us recall some basic properties of the solution $(u,v)$  constructed in \cite{INSZ2}:
\begin{itemize}
 \item[(H1)] $u>0,v<0, u+3v<0$, $u^2+3v^2<s_+^2/3$ for $r\in(0,\infty)$;
 \item[(H2)] For $t<1/3$: $v> -\frac{s_+}{6}>-\frac16$;
 \item[(H3)] For $t>1/3$: $v< -\frac{s_+}{6}<-\frac16$;
 \item[(H4)] For $t=1/3$: $v= -\frac{s_+}{6}=-\frac16$;
 \item[(H5)] $u'>0,~v'(1+6v)\le0$.
\end{itemize}
\begin{Remark}
In fact, \text{(H5)} was not verified in \cite{INSZ2}.
In the appendix, we will present a proof by using (H1)-(H4)
and the fact that $(u,v)$ is a minimizer of (\ref{energy:radi}).
\end{Remark}

The main result of this paper is stated as follows.

\begin{theorem}\label{thm:main}
Let $(u,v)$ be a stable critical point of (\ref{energy:radi}) for $k=\pm1$ with the properties (H1)-(H4). Then the solution $Q=u(r)F_1+v(r)F_2$ is a local minimizer of Landau-de Gennes energy (\ref{energy:LG}).
That is, for any perturbation $V\in H^1(\mathbb{R}^2, \mathcal{Q})$, it holds
\begin{align}\label{ineq:main}
\mathcal{I}(V) &\triangleq\frac{d^2}{d\ve^2}\int_{\BR}\Big\{\frac{1}{2}|\nabla(Q+\ve V)|^2-\frac{1}{2}|\nabla Q|^2
-\frac{t}{2}(|Q+\ve V|^2-|Q|^2)\\
\nonumber&\qquad\qquad-\frac{1}{3}\big[\mathrm{tr}((Q+\ve V)^3)-\mathrm{tr}(Q^3)\big]+\frac{1}{4}(|Q+\ve V|^4-|Q|^4)\Big\}\ud x\Big|_{\ve=0} \\
\nonumber&=\int_{\BR}\Big\{|\nabla V|^2-t|V|^2-2\mathrm{tr}(QV^2)+|Q|^2|V|^2+2(\mathrm{tr}(QV))^2\Big\}\ud x\ge 0.
\end{align}
Moreover, the equality holds if and only if $V\in\text{span}\{V_0, V_1, V_2, V_3, V_4\}$, where
\begin{align*}
V_0&=uE_2,\\
V_1&=(u'(r)E_1+\sqrt{3}v'(r)E_0)\cos\varphi-\frac{ku}{r}E_2\sin\varphi,\\
V_2&=(u'(r)E_1+\sqrt{3}v'(r)E_0)\sin\varphi+\frac{ku}{r}E_2\cos\varphi,\\
V_3&=(u\cos k\varphi-3v)E_3+u\sin k\varphi E_4,\\
V_4&=u\sin k\varphi E_3 -(u\cos k\varphi+3v)E_4.
\end{align*}
Here $E_0, E_1,\cdots, E_4$ are defined in (\ref{def:Ei}).
\end{theorem}
\begin{Remark}
The above null space is generated by the invariance
of the energy $F_{LG}$ under the rotation and translation, i.e.,
\begin{align}
\mathcal{F}_{LG}(RQR^T)=\mathcal{F}_{LG}(Q),\quad \mathcal{F}_{LG}\big(Q(x+x_0)\big)=\mathcal{F}_{LG}(Q(x)),
\end{align}
for any constant $R\in SO(3)$ and constant $x_0\in\mathbb{R}^2.$
\end{Remark}

\begin{Remark}
This result implies that the solution (\ref{prsolu}) for $k=\pm1$ can be regarded as the profile of point defects in $\mathbb{R}^2$ or
the local profile of line defects in $\mathbb{R}^3$. 
\end{Remark}

\setcounter{equation}{0}
\section{The second variation of Landau-de Gennes energy}

To prove the stabillity, we need to compute the second variation of
$\mathcal{F}_{LG}$ at critical point $Q=u(r)F_1+v(r)F_2$. For any $V\in H^1(\BR,\mathcal{Q})$, we have
\begin{align}\label{iv}
  \mathcal{I}(V) &\triangleq\frac{d^2}{d\ve^2}\int_{\BR}\Big\{\frac{1}{2}|\nabla(Q+\ve V)|^2-\frac{1}{2}|\nabla Q|^2
-\frac{t}{2}(|Q+\ve V|^2-|Q|^2)\\
\nonumber&\qquad-\frac{1}{3}(\mathrm{tr}((Q+\ve V)^3)-\mathrm{tr}(Q^3))+\frac{1}{4}(|Q+\ve V|^4-|Q|^4)\Big\}\ud x\Big|_{\ve=0} \\
\nonumber&=\int_{\BR}\Big\{|\nabla V|^2-t|V|^2-2tr(QV^2)+c^2|Q|^2|V|^2+2(\mathrm{tr}(QV))^2\Big\}\ud x\\
\nonumber&=\int_{\BR}\Big\{|\nabla V|^2-t|V|^2-2(u\cdot \mathrm{tr}(F_1V^2)+v\cdot \mathrm{tr}(F_2V^2))\\
\nonumber&\qquad+(6v^2+2u^2)|V|^2+2(u\cdot \mathrm{tr}(F_1V)+v\cdot \mathrm{tr}(F_2V))^2\Big\}\ud x.
\end{align}
We define
\begin{equation}\label{def:Ei}
 \begin{split}
  E_0&=\sqrt{\frac{3}{2}}\left(
                                \begin{array}{ccc}
                                  -\frac{1}{3} & 0 & 0 \\
                                  0 & -\frac{1}{3} & 0 \\
                                  0 & 0 & \frac{2}{3} \\
                                \end{array}
                              \right)=\frac{1}{\sqrt{6}}F_2,
\\
   E_1&=\frac{1}{\sqrt{2}}\left(
        \begin{array}{ccc}
          \cos k\varphi & \sin k\varphi & 0 \\
          \sin k\varphi & -\cos k\varphi & 0 \\
          0 & 0 & 0 \\
        \end{array}
      \right)=\frac{1}{\sqrt{2}}F_1,\\
 E_2&=
        \frac{1}{\sqrt{2}}\left(
        \begin{array}{ccc}
          -\sin k\varphi & \cos k\varphi & 0 \\
          \cos k\varphi & \sin k\varphi & 0 \\
          0 & 0 & 0 \\
        \end{array}
      \right),\\
 E_3&=
\frac{1}{\sqrt{2}}\left(
  \begin{array}{ccc}
    0 & 0 & 1 \\
    0 & 0 & 0 \\
    1 & 0 & 0 \\
  \end{array}
\right),\qquad E_4=
\frac{1}{\sqrt{2}}\left(
  \begin{array}{ccc}
    0 & 0 & 0 \\
    0 & 0 & 1 \\
    0 & 1 & 0 \\
  \end{array}
\right).
\end{split}
\end{equation}
A straightforward calculation shows
\begin{equation}
 \mathrm{tr}(E_iE_j)=\delta_j^i\quad \text{for } 0\le i, j\le 4,\nonumber
\end{equation}
which implies that $\{E_i\}_{0\le i\le 4}$ is an orthonormal basis in $\BQ$.
Thus, we can write $V\in H^1(\BR,\mathcal{Q})$ as a linear combination of this basis in the polar coordinate
\begin{equation}\label{lrcb}
V(r,\varphi)=\sum\limits_{i=0}^4 w_i(r,\varphi)E_i(\varphi).
\end{equation}
Using (\ref{lrcb}), a direct calculation yields that
\begin{align*}
  \nonumber |V|^2 =& \sum\limits_{i=0}^4w_i^2,\quad \mathrm{tr}(F_1V) =\sqrt{2}w_1, \quad \mathrm{tr}(F_2V)=\sqrt{6}w_0,\\
 \nonumber V^2=&w_0^2E_0^2+(w_1^2+w_2^2)E_1^2+w_3^2E_3^2+w_4^2E_4^2+(w_1w_3+w_2w_4)(E_1E_3+E_3E_1)\\
\nonumber &+(w_1w_4+w_2w_3)(E_1E_4+E_4E_1)+w_3w_4(E_3E_4+E_4E_3)-\frac{\sqrt{6}}{3}w_0w_1E_1-\frac{\sqrt{6}}{3}w_0
                  w_2E_2\\
\nonumber &+\frac{\sqrt{6}}{6}w_0w_3E_3+\frac{\sqrt{6}}{6}w_0w_4E_4,
\end{align*}
thus we obtain
\begin{align*}
  \mathrm{tr}(F_1V^2)&=\frac{\cos k\va}{2}(w_3^2-w_4^2)-\frac{2\sqrt{3}}{3}w_0w_1+\sin k\va w_3w_4,\\
   \mathrm{tr}(F_2V^2)&=w_0^2-w_1^2-w_2^2+\frac{1}{2}w_3^2+\f{1}{2}w_4^2.
\end{align*}
From the fact $|\nabla V|^2=(\partial_rV)^2+\frac{1}{r^2}(\partial_\varphi V)^2$, we have
\begin{equation}\label{eq:dv}
  |\nabla V|^2=\sum\limits_{i=0}^{4}w_{ir}^2+\frac{1}{r^2}\big(w_{0\varphi}^2
  +(kw_2-w_{1\varphi})^2+(kw_1+w_{2\varphi})^2+w_{3\varphi}^2+w_{4\varphi}^2\big).\nonumber
\end{equation}

In summary, we conclude that
\begin{align}
  \nonumber \mathcal{I}(V) =& \int_0^{+\infty}\int_0^{2\pi}\Big\{ \sum\limits_{i=0}^{4}w_{ir}^2+\frac{1}{r^2}
  \big[w_{0\varphi}^2+(kw_2-w_{1\varphi})^2+(kw_1+w_{2\varphi})^2+w_{3\varphi}^2+w_{4\varphi}^2\big]\\
  \nonumber &+(6v^2+2u^2-t)\big(\sum\limits_{i=0}^4w_i^2\big)-2\big[u(\frac{1}{2}\cos{k\varphi}(w_3^2-w_4^2)-\frac{2}{\sqrt{3}}w_0w_1+\sin{k\varphi}w_3w_4)\\
&+v(w_0^2-w_1^2-w_2^2+\frac{1}{2}w_3^2+\frac{1}{2}w_4^2)\big] +4(uw_1+\sqrt{3}vw_0)^2 \Big\}rdrd\varphi.\nonumber
\end{align}
In order to prove $\mathcal{I}(V)\ge 0$, it suffices to show that
for $|k|=1$,
\begin{align}\label{IA}
  \nonumber I^A(w_0,w_1,w_2) \triangleq&  \int_0^{+\infty}\int_0^{2\pi}\Big\{w_{0r}^2+w_{1r}^2+w_{2r}^2
  +\frac{1}{r^2}[w_{0\varphi}^2+(kw_2-w_{1\varphi})^2+(kw_1+w_{2\varphi})^2]\\
  \nonumber &+(6v^2+2u^2-t)(w_0^2+w_1^2+w_2^2)-\frac{4}{\sqrt3}uw_0w_1\\
&+v(w_0^2-w_1^2-w_2^2)] +4(uw_1+\sqrt{3}vw_0)^2 \Big\}rdrd\varphi\ge 0,
\end{align}
and
\begin{align}\label{IB}
\nonumber I^B(w_3,w_4) \triangleq& \int_0^{+\infty}\int_0^{2\pi}\Big\{ w_{3r}^2+w_{4r}^2+\frac{1}{r^2}(w_{3\varphi}^2+w_{4\varphi}^2)+(6v^2+2u^2-t)(w_3^2+w_{4}^2)\\
 &-u(\cos{k\varphi}(w_3^2-w_4^2)+2\sin{k\varphi}w_3w_4)-v(w_3^2+w_4^2) \Big\}rdrd\varphi\ge 0.
\end{align}

The following lemma shows that $C_c^\infty(\mathbb{R}^2\backslash\{0\})$ is dense in $H^1(\mathbb{R}^2)$. Thus, we may assume  $V\in C_c^\infty(\mathbb{R}^2\backslash\{0\})$, hence $w_i\in C_c^\infty(\mathbb{R}^2\backslash\{0\})$.

\begin{lemma}
$C_c^\infty(\mathbb{R}^2\backslash\{0\})$ is dense in $H^1(\mathbb{R}^2)$.
\end{lemma}

\begin{proof}
Since $C_c^\infty(\mathbb{R}^2)$ is dense in $H^1(\mathbb{R}^2)$,
it suffices to show that any $C_c^\infty(\mathbb{R}^2)$ function can be approximated by $C_c^\infty(\mathbb{R}^2\backslash\{0\})$.
For this, we introduce a smooth cut-off function $\chi(r)$ defined by
\beno
\chi(r)=\left\{
\begin{array}{l}
0\qquad r\le -2\\
1\qquad r\ge -1.
\end{array}\right.
\eeno
For any $u\in C_c^\infty(\mathbb{R}^2)$, let $u_N(x)=u(x)\chi(\frac {\ln |x|} N)$ for $N\ge 1$. Obviously,
$u_N\in C_c^\infty(\mathbb{R}^2\backslash\{0\})$. Moreover,
\begin{align*}
\|u-u_N\|_{H^1}\le& \big\|u(1-\chi(\frac {\ln |x|} N))\big\|_{L^2}+\big\|\na u(1-\chi(\frac {\ln |x|} N))\big\|_{L^2}\\
&+\frac 1 N\big\|\f u {|x|}\chi'(\frac {\ln |x|} N))\big\|_{L^2}.
\end{align*}
It is easy to see that the first two terms on the right hand side tend to zero as $N\rightarrow +\infty$.
While, the third term is bounded by
\beno
\f C N\Big(\int_{e^{-2N}\le |x|\le e^{-N}}\frac {|u(x)|^2} {|x|^2}\ud x\Big)^\f12\le \f C {\sqrt N}\|u\|_{L^\infty},
\eeno
which tends to zero as $N\rightarrow +\infty$.
\end{proof}

\setcounter{equation}{0}
\section{Some important integral identities}

In this section, let us derive some important integral identities, which will  play crucial roles in our proof. In the sequel, we assume that  $\eta\in C_c^\infty((0,+\infty),\mathbb{R})$.

Using (\ref{ODE}), we deduce that
\begin{align}\nonumber
\mathcal{A}(\eta):=&\int_{0}^\infty \Big\{ (v\eta)_{r}^2+(6v^2+2u^2-t-v)(v\eta)^2\Big\}rdr\\\nonumber
=&\int_{0}^\infty \Big\{ (v\eta)_{r}^2+(v''+\frac{v'}{r}-\frac{1}{3}u^2)v\eta^2\Big\}rdr\\\nonumber
=&(vv'\eta^2r)|_0^{\infty}+\int_0^{\infty} (v^2\eta_r^2-\frac{1}{3}vu^2\eta^2)rdr\\
=&\int_{0}^\infty \Big\{ (v\eta_{r})^2-\frac13vu^2\eta^2\Big\}rdr,\label{eq:A}
\end{align}
and
\begin{align}\nonumber
\mathcal{B}(\eta):=&\int_{0}^\infty \Big\{ (u\eta)_{r}^2+(6v^2+2u^2-t+2v+\frac{k^2}{r^2})(u\eta)^2\Big\}rdr\\
  \nonumber =&\int_0^{\infty} [(\eta'u+u'\eta)^2+u\eta^2(u''+\frac{u'}{r})]rdr\\
  \nonumber =&(uu'\eta^2r)|_0^{\infty} +\int_0^{\infty} u^2\eta_r^2rdr\\
=&\int_{0}^\infty (u\eta_{r})^2rdr.\label{eq:B}
\end{align}

Taking derivative to (\ref{ODE}) gives
\begin{eqnarray*}
&&u'''+\frac{u''}{r}-\frac{2u'}{r^2}+\frac{2u}{r^3}=u'[-t+2v+6v^2+6u^2]+ 2uv'(1+6v), \\
&&v'''+\frac{v''}{r}-\frac{v'}{r^2}=v'[-t-2v+18v^2+2u^2]+\frac{2uu'}{3}(1+6v).
\end{eqnarray*}
Therefore, we have
\begin{align}
\mathcal{C}(\eta)&:=\int_0^{\infty}\Big\{(v'\eta)_r^2 +(v'\eta)^2(18v^2+2u^2-t-2v)\Big\} rdr\nonumber\\
&=\int_0^{\infty}\Big\{(v''\eta+v'\eta')^2 +
v'\eta^2(v'''+\frac{v''}{r}-\frac{v'}{r^2}-\frac{2uu'}{3}(1+6v))\Big\} rdr\nonumber\\\nonumber
&=\int_0^{\infty}\Big\{(v'\eta')^2 -\frac{(v'\eta)^2}{r^2}
-\frac{2uu'v'(1+6v)}{3}\eta^2\Big\} rdr+(rv''v'\eta^2)\big|_{0}^{\infty}\\
&=\int_0^{\infty}\Big\{(v'\eta')^2 -\frac{(v'\eta)^2}{r^2}
-\frac{2uu'v'(1+6v)}{3}\eta^2\Big\} rdr,\label{eq:C}
\end{align}
and
\begin{align}
\mathcal{D}(\eta)&:=\int_0^{\infty}\Big\{(u'\eta)_r^2 +
(u'\eta)^2(6v^2+6u^2-t+2v+\frac{k^2}{r^2})\Big\} rdr\nonumber\\
&=\int_0^{\infty}\Big\{(u''\eta+u'\eta')^2 +
u'\eta^2\Big(u'''+\frac{u''}{r}-\frac{u'}{r^2}+\frac{2u}{r^3}-{2uv'}(1+6v)\Big)\Big\} rdr\nonumber\\\nonumber
&=\int_0^{\infty}\Big\{(u'\eta')^2 -
\frac{(u'\eta)^2}{r^2}+{\frac{2uu'\eta^2}{r^3}}-2uu'v'(1+6v)\eta^2\Big\} rdr+(ru''u'\eta^2)\big|_{0}^{\infty}\\
&=\int_0^{\infty}\Big\{(u'\eta')^2 -
\frac{(u'\eta)^2}{r^2}+{\frac{2uu'\eta^2}{r^3}}-2uu'v'(1+6v)\eta^2\Big\} rdr.\label{eq:D}
\end{align}
In addition, we have
\begin{align}\nonumber
\mathcal{E}(\eta):=&\int_{0}^\infty \Big\{ (\frac{u\eta}{r})_{r}^2+(6v^2+2u^2-t+2v+\frac{k^2}{r^2})\frac{(u\eta)^2}{r^2}\Big\}rdr\\
  \nonumber =&\int_0^{\infty} [(\frac{u\eta'}{r}+(\frac{u}{r})'\eta)^2+\frac{u\eta^2}{r^2}(u''+\frac{u'}{r})]rdr\\\nonumber
=&\int_{0}^\infty\Big( (\frac{u}{r}\eta_{r})^2+\frac{\eta^2}{r^4}(2ruu'-u^2)\Big)rdr+(\frac{u}{r})'u\eta^2\big|_0^\infty\\
=&\int_{0}^\infty\Big( (\frac{u}{r}\eta_{r})^2+\frac{\eta^2}{r^4}(2ruu'-u^2)\Big)rdr.\label{eq:E}
\end{align}

\setcounter{equation}{0}
\section{Proof of Theorem \ref{thm:main}}

This section is devoted to the proof of  Theorem \ref{thm:main}.

\subsection{Non-negativity of $I^B(w_3,w_4)$}

\begin{proposition}\label{prop:A}
For any $w_3, w_4\in C_c^\infty(\mathbb{R}^2\backslash\{0\})$ and  $|k|=1$, we have
\begin{align}\label{IB1}
\nonumber I^B(w_3,w_4):=& \int_0^{+\infty}\int_0^{2\pi}\Big\{ w_{3r}^2+w_{4r}^2
+\frac{1}{r^2}(w_{3\varphi}^2+w_{4\varphi}^2)+(6v^2+2u^2-t)(w_3^2+w_{4}^2)\\
 &-u(\cos{k\varphi}(w_3^2-w_4^2)+2\sin{k\varphi}w_3w_4)-v(w_3^2+w_4^2) \Big\}rdrd\varphi\geq 0.\nonumber
\end{align}
\end{proposition}
\begin{proof}
Let $z=w_3+iw_4$, $i=\sqrt{-1}$. Then we have
\begin{align*}
&|\pa_rz|^2 = (\pa_rw_3)^2+(\pa_rw_4)^2,\quad |\pa_\varphi z|^2
= (\pa_\varphi w_3)^2+(\pa_\varphi w_4)^2,\quad |z|^2 = w_3^2+w_4^2, \\
 & \cos k\varphi(w_3^2-w_4^2)+2\sin k\varphi w_3w_4
 = Re(\cos k\varphi-i\sin k\varphi)(w_3+iw_4)^2=Re(e^{-ik\varphi}z^2).
\end{align*}
Thus, we can rewrite $I^B(w_3,w_4)$ as
\begin{equation}\label{i34}
  I^B(z)=\int_0^{2\pi}\int_0^{\infty}|\pa_rz|^2+\f{1}{r^2}|\pa_{\varphi}z|^2+(6v^2+2u^2-t-v)|z|^2-uRe(e^{-ik\varphi}z^2)rdrd\varphi.
\end{equation}
Assume that
\begin{equation}
z(r,\varphi)=\sum\limits_{m=-\infty}^{+\infty}z_m(r)e^{im\varphi},\nonumber
\end{equation}
we have
\begin{align}
  \nonumber Re(e^{-ik\varphi}z^2) =& Re\big(e^{-ik\varphi}\sum\limits_{l,m=-\infty}^{+\infty}e^{(m+l)\varphi}z_mz_l\big)
   = Re\big(\sum\limits_{l,m=-\infty}^{+\infty}e^{i(m+l-k)\varphi}z_mz_l\big).
\end{align}
 Substituting it into (\ref{i34}), we get
\begin{align*}
I^B(z)=2\pi\int_0^{+\infty}&\Big\{\sum\limits_{m=-\infty}^{+\infty}\Big[|\pa_rz_m|^2
+\f{m^2}{r^2}|z_m|^2+(6v^2+2u^2-t-v)|z_m|^2\Big]\\
&-u\sum\limits_{m+l=k}Re(z_mz_l)\Big\}rdr.
\end{align*}
When $k=1$, we can write
\begin{align*}
I^B(z)=2\pi\sum\limits_{m=1}^{\infty}M_m,
\end{align*}
where
\begin{align}\nonumber
M_m&~=\int_0^\infty\Big\{|\pa_rz_m|^2+|\pa_rz_{1-m}|^2+\frac1{r^2}(m^2|z_m|^2+(1-m)^2|z_{1-m}|^2)\\
&\qquad\qquad+(6v^2+2u^2-t-v)(|z_m|^2+|z_{1-m}|^2)-2uRe(z_mz_{1-m})\Big\}rdr.\nonumber
\end{align}
Noticing that $m^2\ge 1,(1-m)^2\ge0$ for $m\ge1$, and using the following simple relations
\begin{eqnarray}
 |z_mz_{1-m}| \geq  Re(z_mz_{1-m}),\quad |\pa_rz_m|^2 \geq (\pa_r|z_m|)^2, \quad |\pa_rz_{1-m}|^2 \geq (\pa_r|z_{1-m}|)^2,\nonumber
\end{eqnarray}
we conclude that
\begin{align*}
M_m&\ge \int_0^\infty\Big((\pa_r|z_m|)^2+(\pa_r|z_{1-m}|)^2+\frac1{r^2}|z_m|^2\\
&\qquad\qquad+(6v^2+2u^2-t-v)(|z_m|^2+|z_{1-m}|^2)-2u|z_m||z_{1-m}|\Big) rdr.
\end{align*}
Thus, we only need to show that for any $q_0 ,q_1\in C_c^\infty((0,\infty),\mathbb{R})$,
\begin{equation}
\tilde{I}(q_0,q_1)\triangleq \int_0^{\infty}\Big\{(\pa_r q_0)^2+(\pa_r q_1)^2
+\frac{q_1^2}{r^2}+(6v^2+2u^2-t-v)(q_0^2+q_1^2)-2uq_0q_1\Big\}rdr> 0.
\nonumber
\end{equation}
Let $\eta=q_1/u$ and $\zeta=q_0/v$. Then $\eta,\zeta\in C_c^{\infty}((0,\infty))$.
From (\ref{eq:A}) and (\ref{eq:B}), it is straightforward to obtain
\begin{align}\label{ineq:B}\nonumber
\tilde{I}(q_0,q_1)& = \mathcal{A}(\zeta)+\mathcal{B}(\eta)-\int_0^{\infty}\Big\{3v(u\eta)^2+2vu^2\zeta\eta\Big\}rdr\\
&=\int_0^{\infty}\Big\{(u\eta')^2+(v\zeta')^2-\frac13vu^2(3\eta+\zeta)^2\Big\} rdr,
\end{align}
which is non-negative since $v<0$.

The case of $k=-1$ can be considered similarly.
\end{proof}

\subsection{Non-negativity of $I^A(w_0,w_1,w_2)$}
First of all, we expand $w_i(r,\varphi)$ as
\begin{equation}
w_i(r,\varphi)=\sum\limits_{n=0}^{\infty}\big( \mu_n^{(i)}(r)\cos{n\varphi}+\nu_n^{(i)}\sin{n\varphi}\big).\nonumber
\end{equation}
Since $\om_i\in C_c^\infty(\mathbb{R}^2\backslash\{0\})$, we may assume $\mu_n^{(i)},\nu_n^{(i)}\in C_c^\infty((0,\infty),\mathbb{R})$ for all $n$ and $i$.  Furthermore, $w_0\in C_c^\infty(\mathbb{R}^2)$ requires $\mu_n^{(0)}=\nu_n^{(0)}=0$ for $n\ge 1$.
\medskip

Direct calculation shows that
\begin{align*}
&\int_{0}^{2\pi}(\frac{\partial w_i}{\partial r})^2d\varphi =\pi\big[2(\frac{\partial \mu_0^{(i)}}{\partial r})^2
+\sum\limits_{n=1}^{\infty}((\frac{\partial \mu_n^{(i)}}{\partial r})^2+(\frac{\partial \nu_o^{(i)}}{\partial r})^2)\big],\\
&\int_{0}^{2\pi} (\frac{\partial w_i}{\partial \varphi})^2d\varphi=\pi\sum\limits_{n=1}^{\infty}n^2((\mu_n^{(i)})^2+(\nu_n^{(i)})^2),\\
&\int_{0}^{2\pi}\frac{\partial w_i}{\partial \varphi}w_jd\varphi =\pi\sum\limits_{n=1}^{\infty}n(\mu_n^{(j)}\nu_n^{(i)}-\nu_n^{(j)}\mu_n^{(i)}),\\
&\int_{0}^{2\pi} w_i^2d\varphi =\pi\big[2(\mu_0^{(i)})^2+\sum\limits_{n=1}^{\infty}((\mu_n^{(i)})^2+(\nu_n^{(i)})^2)\big],\\
&\int_{0}^{2\pi} w_iw_jd\varphi=\pi\big[2\mu_0^{(i)}\mu_0^{(j)}+\sum\limits_{n=1}^{\infty}(\mu_n^{(i)}\mu_n^{(j)}+\nu_n^{(i)}\nu_n^{(j)})\big].
\end{align*}
Thus, we can decompose $I^A(w_0,w_1,w_2)$ as
\begin{equation}
I^A(w_0,w_1,w_2) =I^A_{0,01}+I^A_{0,2}+\sum\limits_{n=1}^{\infty} I^A_n,
\end{equation}
where
\begin{align*}
I^A_{0,01}= & \int_0^{\infty}\Big\{(\frac{\partial \mu_0^{(0)}}{\partial r})^2 +(\mu_0^{(0)})^2(18v^2+2u^2-t-2v)+(\frac{\partial \mu_0^{(1)}}{\partial r})^2\\
  \nonumber &+(\mu_0^{(1)})^2(6v^2+6u^2-t+2v+\frac{k^2}{r^2})+\frac{4u}{\sqrt3}(1+6v)\mu_0^{(0)}\mu_0^{(1)} rdr,\\
I^A_{0,2}=&\int_0^\infty (\frac{\partial \mu_0^{(2)}}{\partial r})^2+ (\mu_0^{(2)})^2(6v^2+2u^2-t+2v+\frac{k^2}{r^2})\Big\}rdr,
\end{align*}
and
\begin{align*}
I^A_n=&\int_0^{\infty}\Big\{\sum\limits_{i=0}^2\Big((\frac{\partial \mu_n^{(i)}}{\partial r})^2+(\frac{\partial \nu_n^{(i)}}{\partial r})^2\Big)
+\frac{4kn}{r^2}(\mu_n^{(1)}\nu_n^{(2)}-\mu_n^{(2)}\nu_n^{(1)})+\sum\limits_{i=0}^2\frac{n^2}{r^2}((\mu_n^{(i)})^2+(\nu_n^{(i)})^2)\\
\nonumber&+((\mu_n^{(0)})^2+(\nu_n^{(0)})^2)(18v^2+2u^2-t-2v)+((\mu_n^{(1)})^2+(\nu_n^{(1)})^2)(6v^2+6u^2-t+2v+\frac{k^2}{r^2})\\
\nonumber&+((\mu_n^{(2)})^2+(\nu_n^{(2)})^2)(6v^2+2u^2-t+2v+\frac{k^2}{r^2})
+\frac{4u}{\sqrt3}(1+6v)(\mu_n^{(0)}\mu_n^{(1)}+\nu_n^{(0)}\nu_n^{(1)})\Big\}rdr.
 \end{align*}

The non-negativity of $I^A_{0,01}$ follows from the fact that $(u,v)$ is a local minimizer of reduced energy (\ref{energy:radi}). Indeed, we have
 for any $\eta,\xi\in H^1((0,\infty),rdr)$,
\begin{align}
\mathcal{J}(\eta,\xi) &\triangleq\frac{d^2}{d\ve^2}\Big\{{E}(u+\ve\eta,v+\ve\xi)-{E}(u,v)\Big\}\Big|_{\ve=0}\nonumber\\
\nonumber&=\int_0^{\infty}\Big\{(\partial_r\eta)^2+(\partial_r\xi)^2+\eta^2(18v^2+2u^2-t-2v)\\
  \nonumber &\qquad\qquad+\xi^2(6v^2+6u^2-t+2v+\frac{k^2}{r^2})+\frac{4u}{\sqrt3}(1+6v)\eta\xi\Big\} rdr\geq 0,\nonumber
\end{align}
which implies that
\ben
I^A_{0,01}\ge 0.\label{ineq:A001}
\een

It follows from  (\ref{eq:B}) that
\begin{align}\label{ineq:A-0}
I^A_{0,2}=\mathcal{B}(\mu_0^{(2)}/u)=\int_{0}^\infty \Big(u\partial_r(\mu_0^{(2)}/u)\Big)^2rdr\ge 0.
\end{align}

It remains to  prove $I^A_n\ge0$ for all $n\ge1$, which is a consequence
of the following proposition.

\begin{proposition}\label{prop:B}
For any $ \mu_0, \nu_0, \mu_1, \nu_1, \mu_2, \nu_2\in C_c^\infty((0,\infty))$ and  $|k|=1$, we have
\begin{align*}
&I^A_n( \mu_0, \nu_0, \mu_1, \nu_1, \mu_2, \nu_2)\\
&\triangleq\int_0^{\infty}\Big\{\sum\limits_{i=0}^2\Big((\frac{\partial \mu_i}{\partial r})^2+(\frac{\partial \nu_i}{\partial r})^2\Big)
+\frac{4kn}{r^2}(\mu_1\nu_2-\mu_2\nu_1)+\sum\limits_{i=0}^2\frac{n^2}{r^2}(\mu_i^2+\nu_i^2)\\
\nonumber&\qquad+(\mu_0^2+\nu_0^2)(18v^2+2u^2-t-2v)+(\mu_1^2+\nu_1^2)(6v^2+6u^2-t+2v+\frac{k^2}{r^2})\\
\nonumber&\qquad+(\mu_2^2+\nu_2^2)(6v^2+2u^2-t+2v+\frac{k^2}{r^2})
+\frac{4u}{\sqrt3}(1+6v)(\mu_0\mu_1+\nu_0\nu_1)\Big\}rdr\ge 0.
 \end{align*}
\end{proposition}
\begin{proof} From the fact that
\begin{align*}
2(\mu_1\nu_2-\mu_2\nu_1) \ge -(\mu_1^2+\nu_1^2+\mu_2^2+\nu_2^2),
\end{align*}
and $n\ge 1$, we get
\begin{align*}
&4n(\mu_1\nu_2-\mu_2\nu_1)+n^2(\mu_1^2+\nu_1^2+\mu_2^2+\nu_2^2)\ge 4(\mu_1\nu_2-\mu_2\nu_1)+(\mu_1^2+\nu_1^2+\mu_2^2+\nu_2^2).
\end{align*}
So, it suffices to consider the case of $n=1$.

On the other hand, we have
\begin{align*}
|\mu_0\mu_1+\nu_0\nu_1|\le \sqrt{\mu_0^2+\nu_0^2}\sqrt{\mu_1^2+\nu_1^2},\\
|\mu_1\nu_2-\mu_2\nu_1|\le \sqrt{\mu_2^2+\nu_2^2}\sqrt{\mu_1^2+\nu_1^2},
\end{align*}
and $(\partial_r\mu)^2+(\partial_r\nu)^2\ge(\partial_r\sqrt{\mu^2+\nu^2})^2$. Thus, we only need to prove that for $\alpha_i=\pm\sqrt{\mu_{i}^2+\nu_{i}^2}$,
\begin{align*}
\tilde{I}^A_1(\alpha_0, \alpha_1,\alpha_2)=&\int_0^{\infty}\Big((\partial_r\alpha_0)^2+(\partial_r\alpha_1)^2+(\partial_r\alpha_2)^2
-\frac{4}{r^2}\alpha_1\alpha_2+\frac{1}{r^2}(\alpha_0^2+\alpha_1^2+\alpha_2^2)\\
\nonumber&\qquad+\alpha_0^2(18v^2+2u^2-t-2v)+\alpha_1^2(6v^2+6u^2-t+2v+\frac{1}{r^2})\\
\nonumber&\qquad+\alpha_2^2(6v^2+2u^2-t+2v+\frac{1}{r^2})
+\frac{4u}{\sqrt3}(1+6v)\alpha_0\alpha_1\Big)rdr\geq0.
\end{align*}
Let $\xi=\alpha_0/v'$, $\eta=\alpha_1/u'$, $\zeta=r\alpha_2/u$.
Then we infer from (\ref{eq:C})-(\ref{eq:E}) that
\begin{align}\nonumber
&\tilde{I}^A_1(\alpha_0, \alpha_1,\alpha_2)\\\nonumber
&=\mathcal{C}(\xi)+\mathcal{D}(\eta)+\mathcal{E}(\zeta)\\
&\quad+\int_0^{\infty}\Big(-\frac{4}{r^3}uu'\eta\zeta+\frac{1}{r^2}((v'\xi)^2+(u'\eta)^2+(\frac{u\zeta}{r})^2)
+\frac{4uu'v'}{\sqrt3}(1+6v)\xi\eta\Big)rdr\nonumber\\\label{ineq:A-1}
&=\int_0^{\infty}\Big((v'\xi')^2+(u'\eta')^2+(\frac{u\zeta'}{r})^2+\frac{2uu'}{r^3}(\eta-\zeta)^2
-\frac{2uu'v'}{\sqrt3}(1+6v)\big(\eta+\frac{\xi}{\sqrt{3}}\big)^2\Big)rdr\geq 0.\end{align}
This completes our proof.
\end{proof}

\subsection{Proof of Theorem \ref{thm:main}}

In order to prove $\mathcal{I}(V)\ge 0$, it suffices to show that
\beno
I^A(w_0,w_1,w_2)\ge 0\quad\text{and}\quad I^B(w_3,w_4)\ge 0,
\eeno
which follow from Proposition \ref{prop:A}, \eqref{ineq:A001}, (\ref{ineq:A-0}) and Proposition \ref{prop:B}.

The second part of Theorem \ref{thm:main} can be directly deduced from (\ref{ineq:B}), (\ref{ineq:A-0}) and (\ref{ineq:A-1}).

\section{Appendix: proof of (H5)}

In this section, we will verify the property (H5) for the solution $(u,v)$ constructed in \cite{INSZ2}.  Let
\begin{align*}
p(r)=uu',\quad q(r)=-v'(1+6v).
\end{align*}
First of all, we will show that $p(r)$ and $q(r)$ are nonnegative.
Recall that the fact that $(u,v)$ is a local minimizer of reduced energy (\ref{energy:radi})  implies that
\begin{align*}
I^A_{0,01}= & \int_0^{\infty}\Big\{(\frac{\partial \mu_0^{(0)}}{\partial r})^2 +(\mu_0^{(0)})^2(18v^2+2u^2-t-2v)+(\frac{\partial \mu_0^{(1)}}{\partial r})^2\\
  \nonumber &+(\mu_0^{(1)})^2(6v^2+6u^2-t+2v+\frac{k^2}{r^2})+\frac{4u}{\sqrt3}(1+6v)\mu_0^{(0)}\mu_0^{(1)}\Big\} rdr\ge 0.\nonumber
\end{align*}
If $\{p(r)<0 \} \cup\{q(r)<0\}\neq \emptyset$, we let
\begin{align*}
\chi =\mathbf{1}_{\{p(r)<0\}},\qquad \eta =-\sqrt{3}\mathbf{1}_{\{q(r)<0\}}.
\end{align*}
Take $\mu_0^{(0)}=v'\eta$ and $\mu_0^{(1)}=u'\chi$. Formally, it follows from (\ref{eq:C}) and (\ref{eq:D})  that
\begin{align}
0\le I^A_{0,01}&=\int_0^{\infty}\Big\{(v'\eta')^2 -\frac{(v'\eta)^2}{r^2}
-\frac{2uu'v'(1+6v)}{3}\eta^2\Big\} rdr+(rv''v'\eta^2)\big|_{0}^{\infty}\nonumber\\\nonumber
&\quad+\int_0^{\infty}\Big\{(u'\chi')^2 -
\frac{(u'\chi)^2}{r^2}+{\frac{2uu'\chi^2}{r^3}}-2uu'v'(1+6v)\chi^2\Big\} rdr+(ru''u'\chi^2)\big|_{0}^{\infty}\\\nonumber
&\quad+\int_0^{\infty}\Big\{\frac{4}{\sqrt3}uu'v'(1+6v)\chi\eta\Big\} rdr\\\nonumber
&=\int_0^{\infty}\Big\{(v'\eta')^2 -\frac{(v'\eta)^2}{r^2}
-\frac{2uu'v'(1+6v)}{3}(\frac{\eta}{\sqrt3}+\chi)^2\Big\} rdr+(rv''v'\eta^2)\big|_{0}^{\infty}\\\nonumber
&\quad+\int_0^{\infty}\Big\{(u'\chi')^2 -
\frac{(u'\chi)^2}{r^2}+{\frac{2uu'\chi^2}{r^3}}\Big\} rdr+(ru''u'\chi^2)\big|_{0}^{\infty}\\\nonumber
&=\int_0^{\infty}\Big\{(v'\eta')^2 -\frac{(v'\eta)^2}{r^2}
+\frac{2pq}{3}(\frac{\eta}{\sqrt3}+\chi)^2+(u'\chi')^2 -
\frac{(u'\chi)^2}{r^2}+{\frac{2p\chi^2}{r^3}}\Big\} rdr.
\end{align}
Notice that
\begin{align*}
&\int_0^{\infty}(v'\eta')^2 rdr =\int_0^{\infty}(u'\chi')^2rdr=0,\\
&-\int_0^{\infty}\Big\{ \frac{(v'\eta)^2}{r^2}+ \frac{(u'\chi)^2}{r^2}\Big\} rdr<0,\\
&\int_0^{\infty}pq(\frac{\eta}{\sqrt3}+\chi)^2 rdr=\int_{\{p<0,q<0\}}pq(\frac{\eta}{\sqrt3}+\chi)^2 rdr
+\int_{\{p<0,q>0\}}pq(\frac{\eta}{\sqrt3}+\chi)^2 rdr\\
&\qquad+\int_{\{p>0,q<0\}}pq(\frac{\eta}{\sqrt3}+\chi)^2 rdr \le 0,\\
&\int_0^{\infty}{\frac{2p\chi^2}{r^3}} rdr=\int_{\{p<0\}}{\frac{2p\chi^2}{r^3}} rdr\le 0,
\end{align*}
which contradict with  $I^A_{0,01} \ge 0.$ Thus, we deduce that
\beno
u'(r)\ge 0,\quad v'(r)(1+6v)\le 0\quad \text{for } r>0.
\eeno
Thanks to the fact that $p(r)\ge 0, q(r)\ge 0$ for $r$ small, and $u'=0$ on $p(r)=0, v'=0$ on $q(r)=0$, the above formal derivation can be justified by a standard smoothing procedure and cutoff argument.

Next, we prove that $u'(r)>0$ for $r>0$.  Otherwise, $u'(r_0)=0$ for some $r_0>0$. Then we have $u''(r_0)=0$ due to $u'(r)\ge 0$. On the other hand, we have
\begin{eqnarray*}
  u'''+\frac{u''}{r}-\frac{2u'}{r^2}+\frac{2u}{r^3}=u'\big(-t+2v+6v^2+6u^2\big)+ 2uv'(1+6v).
\end{eqnarray*}
Taking $r=r_0$, we get
\begin{eqnarray*}
  u'''(r_0)=u(-\frac{2}{r^3}+2v'(1+6v))<0,
\end{eqnarray*}
which contradicts with $u'(r)\ge 0$. Thus, $u'>0$ for all $r>0$.

By a similar argument and the equation
\begin{eqnarray*}
  v'''+\frac{v''}{r}-\frac{v'}{r^2}=v'\big(-t-2v+18v^2+2u^2\big)+\frac{2uu'}{3}(1+6v),
\end{eqnarray*}
we can deduce that  $v$ is also strictly monotonic.

\section*{Acknowledgments}
W. Wang is partly supported by NSF of China under Grant 11501502.
P. Zhang is partly supported by NSF of China under Grant 11421101 and 11421110001.
Z. Zhang is partly supported by NSF of China under Grant 11371039 and 11425103.

\end{document}